\documentclass[a4paper,twoside,10pt,reqno]{amsart}
\usepackage[cm,headings]{fullpage}

\usepackage[utf8]{inputenc}
\usepackage[T1]{fontenc}

\usepackage{color}
\usepackage{graphicx}
\usepackage[pdf,matrix,arrow,curve,frame,color]{xy} 

\usepackage[final]{hyperref}

\newtheorem{theorem}{Theorem}
\newtheorem{corollary}[theorem]{Corollary}
\newtheorem{proposition}[theorem]{Proposition}
\newtheorem{lemma}[theorem]{Lemma}
\theoremstyle{definition}
\newtheorem{definition}[theorem]{Definition}
\theoremstyle{remark}
\newtheorem{question}{Question}
\newtheorem{remark}[theorem]{Remark}

\newcommand\setsep{;\ } 
\newcommand\mc{\mathcal}

\newcommand\abs[1]{\mathopen|#1\mathclose|}
\newcommand\absa[1]{\left|#1\right|}
\newcommand\absb[2]{\csname#1l\endcsname|#2\csname#1r\endcsname|}
\newcommand\norm[1]{\mathopen\|#1\mathclose\|}
\newcommand\norma[1]{\left\|#1\right\|}
\newcommand\normb[2]{\csname#1l\endcsname\|#2\csname#1r\endcsname\|}
\newcommand\tnorm[1]{\mathopen{|\mskip-0.5\thinmuskip|\mskip-0.5\thinmuskip|}#1\mathclose{|\mskip-0.5\thinmuskip|\mskip-0.5\thinmuskip|}}

\newcommand\ev[2]{\langle #1,#2\rangle}

\newcommand\de{\delta}
\newcommand\ve{\varepsilon}
\newcommand\ga{\gamma}
\newcommand\Ga{\Gamma}

\newcommand{\Id}{I\!d}

\newcommand\N{{\mathbb N}}
\newcommand\R{{\mathbb R}}

\newcommand\restr[1]{\mathclose\restriction_{#1}}

\DeclareMathOperator{\suppo}{supp_o}
\DeclareMathOperator{\spn}{span}
\DeclareMathOperator{\cspan}{\overline{span}}
\DeclareMathOperator{\sgn}{sgn}

\newcommand\Lip{\mathrm{Lip}}

\let\restriction=|
\let\comp=\circ

\begin{document}
\title{Isometric embedding of $\ell_1$ into Lipschitz-free spaces and $\ell_\infty$ into their duals}
\author{Marek Cúth}
\address{Department of Mathematical Analysis\\Charles University\\Sokolovská~83\\186~75 Praha~8\\Czech Republic}
\email{cuth@karlin.mff.cuni.cz}
\author{Michal Johanis}
\email{johanis@karlin.mff.cuni.cz}
\date{December 2015}
\thanks{
M.~Cúth is a junior researcher in the University Centre for Mathematical Modelling, Applied Analysis and Computational Mathematics (MathMAC) and was supported by grant P201/12/0290.
M.~Johanis was supported by GAČR~16-07378S.}
\begin{abstract}
We show that the dual of every infinite-dimensional Lipschitz-free Banach space contains an isometric copy of $\ell_\infty$
and that it is often the case that a Lipschitz-free Banach space contains a $1$-complemented subspace isometric to $\ell_1$.
Even though we do not know whether the latter is true for every infinite-dimensional Lipschitz-free Banach space, we show that the space is never rotund.

Further, in the last section we survey the relations between ``isometric embedding of~$\ell_\infty$ into the dual'' and ``containing as good copy of~$\ell_1$ as possible'' in a general Banach space.
\end{abstract}
\maketitle

\section*{Introduction}

Given a metric space $M$ it is possible to construct a Banach space $\mc F(M)$ in such a way that the metric structure of~$M$ corresponds to the linear structure of $\mc F(M)$.
This space $\mc F(M)$ is usually called the Lipschitz-free space over $M$.
The study of Lipschitz-free spaces is well-motivated: using this notion many interesting results have been proved,
e.g. if a separable Banach space~$Y$ is isometric (not necessarily linearly) to a subset of a Banach space~$X$, then $Y$ is already linearly isometric to a subspace of $X$, \cite[Corollary~3.3]{GK}.
We refer to~\cite[Section~1]{CDW} and to the end of this section for some more details concerning the construction and basic properties of these spaces.
Although Lipschitz-free spaces are easy to define, their structure even for separable metric spaces $M$ is poorly understood to this day.
The study of the linear structure of Lipschitz-free spaces over metric spaces has become an active field of study, see~\cite{CDW} and references therein.

The dual $\mc F(M)^*$ is linearly isometric to the space of Lipschitz functions $\Lip_0(M)$.
Recently it was shown in~\cite[Theorem~1]{CDW} that for every infinite metric space $M$ the space $\ell_\infty$ is isomorphic to a subspace of $\Lip_0(M)$,
which is by a classical result of Bessaga and Pełczyński equivalent to the fact that $\mc F(M)$ contains a complemented subspace isomorphic to $\ell_1$.
However, the authors of this result confessed that they do not know whether $\ell_\infty$ embeds isometrically into $\Lip_0(M)$, see~\cite[p.~2]{CDW}.
In this note we show that this is the case.

Let us recall that an ultrametric space is a metric space $(M,\rho)$ such that $\rho(x,y)\le\max\{\rho(x,z),\rho(z,y)\}$ for every $x,y,z\in M$.
Our main result is the following:

\begin{theorem}\label{t:main}
Let $M$ be an infinite metric space.
Then $\Lip_0(M)$ contains a subspace isometric to $\ell_\infty$.
If moreover the completion of $M$ has an accumulation point or contains an infinite ultrametric space, then $\mc F(M)$ contains a $1$-complemented subspace isometric to~$\ell_1$.
\end{theorem}

The first part of this result is contained in Theorem~\ref{t:linfty}, the ultrametric case is contained in Proposition~\ref{p:ultrametric}.
The case of ultrametric spaces is interesting in connection with~\cite[Theorem 2]{CD} and~\cite[Corollary 6]{DKP},
where it is proved that for every separable ultrametric space $M$ the space $\mc F(M)$ is isomorphic to $\ell_1$ but it is never isometric to $\ell_1$.

Our main result is not only a technical improvement of~\cite[Theorem~1]{CDW}.
There are some consequences for the structure of a general Lipschitz-free Banach space which could not be deduced from the isomorphic variant~\cite[Theorem~1]{CDW}, see Corollary~\ref{c:fpp}.
As mentioned above, it is a classical result that for any Banach space $X$ its dual $X^*$ has a subspace isomorphic to $\ell_\infty$ if and only if $\ell_1$ is isomorphic to a complemented subspace of $X$.
In connection with Theorem~\ref{t:main} we are naturally interested in an isometric variant of this result
and we surveyed the relation between ``isometric embedding of $\ell_\infty$ into the dual'' and ``containing as good copy of $\ell_1$ as possible'' for a general Banach space.
We collected several general results which are available at the last section of this note.
Many implications and counterexamples are known; nevertheless, several implications are up to our knowledge unknown.
For example, the following seems to be open:

\begin{question}
Let $X$ be a (separable) Banach space such that $X^*$ has a subspace isometric to $\ell_\infty$.
Does $X$ contain a $1$-complemented subspace isomorphic to $\ell_1$?
\end{question}

It is not true in general that isometric embedding of $\ell_\infty$ into $X^*$ implies isometric embedding of $\ell_1$ into $X$.
Therefore, it is of some interest to find out whether $\ell_1$ embeds isometrically into every Lipschitz-free Banach space over an infinite metric space.
The case when $M$ has an accumulation point is covered by Theorem~\ref{t:main}.

\begin{question}\label{q:discrete}
Let $M$ be a uniformly discrete metric space.
Does $\mc F(M)$ contain a subspace isometric to~$\ell_1$?
\end{question}
In particular, for a few concrete metric spaces for which the answer is not known to us see Remark~\ref{r:unknown_ex}.

Even though we do not know the answer to Question~\ref{q:discrete}, we show that $\mc F(M)$ is never rotund, which supports the conjecture that the answer to Question~\ref{q:discrete} is positive.
Recall that a normed linear space $X$ is rotund if $\norm{x+y}<2$ for any $x,y\in S_X$, $x\neq y$.

\begin{proposition}\label{p:notRotund}
Let $M$ be a metric space of cardinality at least $3$.
Then $\mc F(M)$ is not rotund.
\end{proposition}

This result is interesting also for finite metric spaces:
No Lipschitz-free space over an ultrametric space is isometric to $\ell_1(\Ga)$, see~\cite[Corollary 6]{DKP}.
Hence e.g. if $M$ is the ultrametric space $\{0,x,y\}$ with metric defined by $\rho(x,y)=\rho(x,0)=\rho(y,0)=1$,
then $\mc F(M)$ is not isometric to $\ell_1^2$, however, it is not rotund.

As a consequence of Proposition~\ref{p:notRotund} it follows that the free space norm is quite rare,
as the non-rotund norms are of the first category in the metric space of all equivalent norms, provided that there is at least one equivalent rotund norm on the space; \cite{FZZ}, see \cite[Theorem~II.4.1]{DGZ}.

\medskip
Let us recall some basic facts concerning the Lipschitz-free spaces (for the proofs we refer to~\cite[Section~1]{CDW}).
Let $(M,\rho,0)$ be a pointed metric space, i.e. a metric space with a distinguished ``base point'' denoted by $0$.
Consider the space $\Lip_0(M)$ of all real-valued Lipschitz functions that map $0\in M$ to $0\in\R$.
It has a vector space structure and the minimal Lipschitz constant of $f\in\Lip_0(M)$ given by $\norm f_{\Lip}=\sup\left\{\frac{\abs{f(x)-f(y)}}{\rho(x,y)}\setsep x,y\in M, x\neq y\right\}$ gives rise to a norm on $\Lip_0(M)$.
The space $\smash[t]{\bigl(\Lip_0(M),\norm\cdot_{\Lip}\bigr)}$ is then a Banach space.
For any $x\in M$ we denote by $\de_x\in\Lip_0(M)^*$ the evaluation functional, i.e. $\ev{\de_x}f=f(x)$ for every $f\in\Lip_0(M)$.
Let $\mc F(M)$ be the closure of the linear span of $\{\de_x\setsep x\in M\}$ with the dual space norm denoted simply by $\norm\cdot$.
It is easy to see that $\norm{\de_x-\de_y}=\rho(x,y)$ for any $x,y\in M$.
This means that $M$ can be considered as a metric subspace of $\mc F(M)$ via the isometric embedding $x\mapsto\de_x$.
The space $\mc F(M)$ is usually called the Lipschitz-free Banach space over $M$ and it is uniquely characterised by the following universal property:

Let $X$ be a Banach space and suppose that $L\colon M\to X$ is a Lipschitz mapping satisfying $L(0)=0$.
Then there exists a unique linear operator $\widehat{L}\colon\mc F(M)\to X$ extending $L$, i.e. the following diagram commutes:
\[
\xymatrix{
M \ar[r]^L \ar[d]_(.45){\de} & X\\
\mc F(M) \ar[ur]_{\widehat L}
}
\]
Moreover, $\norm{\widehat L}=\norm L_{\Lip}$, where $\norm L_{\Lip}$ denotes the minimal Lipschitz constant of $L$.

Using this universal property of $\mc F(M)$ for $X=\R$ it can be rather easily shown that $\mc F(M)^*$ is linearly isometric to $\Lip_0(M)$.
It is immediate that the $w^*$ topology on bounded subsets of $\Lip_0(M)$ is the topology of pointwise convergence.

Further, observe that it does not matter how the point $0\in M$ is chosen:
Let $e\in M$ be another base point and denote $N=(M,\rho,e)$.
Then $T\colon\Lip_0(M)\to\Lip_0(N)$ defined by $T(f)=f-f(e)$ is a linear isometry onto that is also $w^*$--$w^*$ continuous (by the Banach-Dieudonné theorem).
Thus $T$ is a dual operator to $T^*\restr{\mc F(N)}\colon\mc F(N)\to\mc F(M)$, which is therefore a linear isometry onto.

Similarly, we may without loss of generality assume that $M$ is complete:
Denote by $N$ the completion of $M$.
Then $T\colon\Lip_0(N)\to\Lip_0(M)$ defined by $T(f)=f\restr M$ is a linear isometry onto that is also $w^*$--$w^*$ continuous.
So, as above, the spaces $\mc F(N)$ and $\mc F(M)$ are linearly isometric.

It is also easy to observe that if $N$ is a subspace of a metric space $M$, then $\mc F(N)$ is linearly isometric to a subspace of $\mc F(M)$.
Using this together with the universal property of $\mc F(M)$ we can see that the metric structure of $M$ corresponds to the linear structure of $\mc F(M)$.
For example, if $N$ is bi-Lipschitz equivalent (resp. isometric) to a subset of $M$, then $\mc F(N)$ is linearly isomorphic (resp. linearly isometric) to a subspace of $\mc F(M)$.

The notation and terminology we use are relatively standard.
If $M$ is a metric space, $x\in M$ and $r\ge0$, we denote by $U(x,r)$ the open ball centred at $x$ with radius $r$.

\section{\texorpdfstring{Isometric embedding of $\ell_\infty$ into $\Lip_0(M)$}{Isometric embedding of l_infty into Lip_0(M)}}

In this section we prove Theorem~\ref{t:linfty} and mention some consequences for the structure of a general Lipschitz-free Banach space, see Corollary~\ref{c:fpp}.
Our main tool throughout is Lemma~\ref{l:seq-isometry}, which allows us to embed isometrically $\ell_\infty$ into $\Lip_0(M)$, resp. $\ell_1$ into $\mc F(M)$.

For a mapping $f\colon X\to Y$, where $X$ is a set and $Y$ a vector space, we denote $\suppo f=f^{-1}(Y\setminus\{0\})$.

\begin{lemma}\label{l:sum-lip}
Let $X$, $Y$ be normed linear spaces, $\Ga$ a set, and let $f_\ga\colon X\to Y$ be $L$-Lipschitz for each $\ga\in\Ga$.
Suppose that the collection $\{\suppo f_\ga\}_{\ga\in\Ga}$ is disjoint.
Then $f=\sum_{\ga\in\Ga}f_\ga$ is $L$-Lipschitz.
\end{lemma}
\begin{proof}
Note that the sum is pointwise finite and so the mapping $f$ is well-defined.
Pick $x,y\in X$.
In case that $f(x)=0$, let $\ga\in\Ga$ be such that $f(y)=f_\ga(y)$.
Then $f_\ga(x)=0$ and hence $\norm{f(x)-f(y)}=\norm{f_\ga(x)-f_\ga(y)}\le L\norm{x-y}$.
Now suppose that $x\in\suppo f_\alpha$, $y\in\suppo f_\beta$ for some $\alpha,\beta\in\Ga$.
Since the line segment $[x,y]$ is connected and $\suppo f_\alpha$, $\suppo f_\beta$ are open and disjoint,
there is $z\in[x,y]\setminus(\suppo f_\alpha\cup\suppo f_\beta)$.
Then
\[\begin{split}
\norm{f(x)-f(y)}&=\norm{f_\alpha(x)-f_\beta(y)}=\norm{f_\alpha(x)-f_\alpha(z)+f_\beta(z)-f_\beta(y)}\le\norm{f_\alpha(x)-f_\alpha(z)}+\norm{f_\beta(z)-f_\beta(y)}\\
&\le L\norm{x-z}+L\norm{z-y}=L\norm{x-y}.
\end{split}\]
\end{proof}

\begin{lemma}\label{l:seq-isometry}
Let $(M,\rho)$ be a metric space such that there are sequences $\{x_n\}\subset M$ and $\{r_n\}\subset[0,+\infty)$ satisfying
$\rho(x_m,x_n)\ge r_m+r_n$ for all $m,n\in\N$, $m\neq n$, and
\[
\lim_{n\to\infty}\frac{r_{2n}+r_{2n+1}}{\rho(x_{2n},x_{2n+1})}=1.
\]
Then $\Lip_0(M)$ contains a subspace isometric to~$\ell_\infty$.

If moreover even $r_{2n}+r_{2n+1}=\rho(x_{2n},x_{2n+1})$ for each $n\in\N$, then $\mc F(M)$ contains a $1$-complemented subspace isometric to~$\ell_1$.
\end{lemma}
\begin{proof}
We may assume that $0=x_1$.
We consider $M$ as a metric subspace of $X=\mc F(M)$.
Set $g_n(x)=\max\{r_n-\norm{x-x_n},0\}$ for $x\in X$, $n\in\N$.
The functions $g_n$ are clearly $1$-Lipschitz, $\{\suppo g_n\}_{n\in\N}=\{U(x_n,r_n)\}_{n\in\N}$ is a disjoint collection, and $g_n(0)=0$ for each $n>1$.
Further, let $\{n^k_l\}_{l=1}^\infty\subset\N$, $k\in\N$ be disjoint increasing sequences.
Finally, we set $f_k=\sum_{l=1}^\infty g_{2n^k_l}-g_{2n^k_l+1}$ in the pointwise convergence on~$X$.
By Lemma~\ref{l:sum-lip}, each $f_k$, $k\in\N$, is a $1$-Lipschitz function.

Now let $(a_k)\in\ell_\infty$ and consider $g=\sum_{k=1}^\infty a_kf_k$ in the pointwise convergence on~$X$.
Then $g$ is $\norm{(a_k)}$-Lipschitz by Lemma~\ref{l:sum-lip} and clearly the same holds for $h=g\restr M$.
Hence $h\in\Lip_0(M)$ and $\norm h\le\norm{(a_k)}$.
On the other hand,
\[\begin{split}
\norm h&\ge\sup_{n\in\N}\frac{\abs{h(x_{2n})-h(x_{2n+1})}}{\rho(x_{2n},x_{2n+1})}=\sup_{n\in\N}\frac{\abs{g(x_{2n})-g(x_{2n+1})}}{\rho(x_{2n},x_{2n+1})}\ge\sup_{k,l\in\N}\frac{\absb{big}{g(x_{2n^k_l})-g(x_{2n^k_l+1})}}{\rho(x_{2n^k_l},x_{2n^k_l+1})}\\
&=\sup_{k,l\in\N}\frac{\absb{big}{a_kr_{2n^k_l}+a_kr_{2n^k_l+1}}}{\rho(x_{2n^k_l},x_{2n^k_l+1})}=\sup_{k\in\N}\abs{a_k}\sup_{l\in\N}\frac{r_{2n^k_l}+r_{2n^k_l+1}}{\rho(x_{2n^k_l},x_{2n^k_l+1})}=\sup_{k\in\N}\abs{a_k}=\norm{(a_k)}.
\end{split}\]
The mapping $(a_k)\mapsto h$ described above is therefore a linear isometry from $\ell_\infty$ into $\Lip_0(M)$.

For the moreover part it suffices to take only $f_n=g_{2n}-g_{2n+1}$.
Further, we set $e_n=\frac{\de_{x_{2n}}-\de_{x_{2n+1}}}{\rho(x_{2n},x_{2n+1})}$.
It is easy to see that $\{(e_n;f_n\restr M)\}$ is a bi-normalised biorthogonal system.
Let us verify that $\{e_n\}$ is $1$-equivalent to the canonical basis of $\ell_1$.
Pick any $a_1,\dotsc,a_N\in\R$ and consider $x=\sum_{n=1}^N a_ne_n$.
Then clearly $\norm x\le\sum_{n=1}^N\abs{a_n}$.
On the other hand, put $g=\sum_{n=1}^N\sgn a_n\cdot f_n\restr M$.
By the above we have $g\in\Lip_0(M)$ and $\norm g\le1$.
Thus, $\norm x\ge\abs{\ev gx}=\absb{big}{\sum_{n=1}^N a_n\ev g{e_n}}=\sum_{n=1}^N\abs{a_n}$.

To finish the proof it remains to find a projection of norm $1$ from $X$ onto $\cspan\{e_n\setsep n\in\N\}$.
Define $r\colon X\to\spn\{e_n\setsep n\in\N\}$ by $r(x)=\sum_{n=1}^\infty f_n(x)e_n$.
Then $r$ is a $1$-Lipschitz mapping by Lemma~\ref{l:sum-lip}.
By the universal property of $X=\mc F(M)$ there is a linear operator $P\colon X\to\cspan\{e_n\setsep n\in\N\}$ such that $P\comp\de=r\restr M$ and $\norm P\le 1$.
In order to see that $P$ is a projection onto $\cspan\{e_n\setsep n\in\N\}$ observe that $P(e_n)=e_n$ for every $n\in\N$:
\[\begin{split}
P(e_n)&=\frac{P(\de_{x_{2n}})-P(\de_{x_{2n+1}})}{\rho(x_{2n},x_{2n+1})}=\frac{r(x_{2n})-r(x_{2n+1})}{\rho(x_{2n},x_{2n+1})}=\frac{f_n(x_{2n})e_n-f_n(x_{2n+1})e_n}{\rho(x_{2n},x_{2n+1})}\\
&=\frac{g_{2n}(x_{2n})e_n+g_{2n+1}(x_{2n+1})e_n}{\rho(x_{2n},x_{2n+1})}=\frac{r_{2n}+r_{2n+1}}{\rho(x_{2n},x_{2n+1})}e_n=e_n.
\end{split}\]
\end{proof}

\begin{theorem}\label{t:linfty}
Let $M$ be an infinite metric space.
Then $\Lip_0(M)$ contains a subspace isometric to $\ell_\infty$.
If moreover the completion of $M$ has an accumulation point, then $\mc F(M)$ contains a $1$-complemented subspace isometric to~$\ell_1$.
\end{theorem}
\begin{proof}
Let $\rho$ be the metric on $M$.
Without loss of generality we may assume that $M$ is complete.
We distinguish three cases: $M$~has an accumulation point, $M$~contains a bounded uniformly separated sequence, and $M$ is unbounded.
This covers all the possibilities, since if $M$ does not have an accumulation point, then it is not totally bounded.

So suppose first that $M$ has an accumulation point.
Then there is a sequence $\{x_n\}_{n=2}^\infty\subset M$ of distinct points converging to some $x_1\in M$.
By passing to a subsequence we may assume that either $\rho(x_m,x_n)=\rho(x_m,x_1)+\rho(x_n,x_1)$ for all $m,n>1$, $m\neq n$,
or $\rho(x_{2n},x_{2n+1})<\rho(x_{2n},x_1)+\rho(x_{2n+1},x_1)$ for all $n\in\N$.
In the first case we can take $r_n=\rho(x_n,x_1)$.
In the second case we put $\de_n=\frac12\bigl(\rho(x_{2n},x_1)+\rho(x_{2n+1},x_1)-\rho(x_{2n},x_{2n+1})\bigr)>0$.
By passing to a subsequence we may assume that $\rho(x_m,x_1)\le\frac12\de_n$ for all $m,n\in\N$, $m>2n+1$.
We set $r_1=0$, $r_{2n}=\rho(x_{2n},x_1)-\de_n$, and $r_{2n+1}=\rho(x_{2n+1},x_1)-\de_n$.
It is easy to see that in both cases the assumptions of Lemma~\ref{l:seq-isometry} are satisfied.
Indeed, in the second case
\[
r_m+r_{2n+i}<\rho(x_m,x_1)+\rho(x_{2n+i},x_1)-\de_n\le\rho(x_m,x_1)+\rho(x_{2n+i},x_1)-2\rho(x_m,x_1)\le\rho(x_m,x_{2n+i})
\]
for $m>2n+1$ and $i\in\{0,1\}$.

Suppose now that $M$ contains a bounded uniformly separated sequence $\{y_n\}$.
Using the boundedness we construct inductively increasing sequences $\{n^k_l\}_{l=1}^\infty\subset\N$ such that $\{n^{k+1}_l\}$ is a subsequence of $\{n^k_l\}$, $n^{k+1}_1>n^k_1$,
and such that $\lim_{l\to\infty}\rho(y_{n^k_1},y_{n^k_l})=d_k$ for each $k\in\N$.
By our assumption the sequence $\{d_k\}$ is bounded and $\inf d_k>0$, hence there is $d\in(0,+\infty)$ and an increasing sequence $\{k_p\}\subset\N$ such that
$d\bigl(1-\frac1{2p}\bigr)<d_{k_p}<d\bigl(1+\frac1{2p}\bigr)$ for each $p\in\N$.
Finally, we inductively find increasing sequences $\{p_s\}\subset\N$, $\{l_s\}\subset\N$ satisfying $p_1=1$, $d\bigl(1-\frac1{2p_s}\bigr)<\rho(y_{n^{k_{p_s}}_1},y_{n^{k_{p_s}}_l})<d\bigl(1+\frac1{2p_s}\bigr)$ for $l\ge l_s$,
and $n^{k_{p_{s+1}}}_1\ge n^{k_{p_s}}_{l_s}$ for each $s\in\N$.
Now we set $x_s=y_{n^{k_{p_s}}_1}$ for $s\in\N$ and note that $d\bigl(1-\frac1{2m}\bigr)<\rho(x_m,x_n)<d\bigl(1+\frac1{2m}\bigr)$ for all $m,n\in\N$, $n>m$.
Therefore if we set $r_n=\frac d2\bigl(1-\frac1n\bigr)$, then $r_m+r_n=d\bigl(1-\frac1{2m}-\frac1{2n}\bigr)<d\bigl(1-\frac1{2m}\bigr)<\rho(x_m,x_n)$ for $n>m$ and
\[
\frac{r_{2n}+r_{2n+1}}{\rho(x_{2n},x_{2n+1})}=\frac{d\bigl(1-\frac1{4n}-\frac1{2(2n+1)}\bigr)}{\rho(x_{2n},x_{2n+1})}>\frac{d\bigl(1-\frac1{4n}-\frac1{2(2n+1)}\bigr)}{d\bigl(1+\frac1{4n}\bigr)}\to1.
\]
Hence again the assumptions of Lemma~\ref{l:seq-isometry} are satisfied.

Finally, assume that $M$ is unbounded.
We construct inductively sequences $\{x_n\}\subset M$ and $\{r_n\}\subset(0,+\infty)$.
Pick any $x_1\in M$ and $r_1>0$.
In the induction step we find $x_{n+1}\in M$ such that $\rho(x_{n+1},x_n)>n\max\{\rho(x_n,x_k)+r_k\setsep k=1,\dotsc,n\}$ and put $r_{n+1}=\rho(x_{n+1},x_n)-\max\{\rho(x_n,x_k)+r_k\setsep k=1,\dotsc,n\}$.
Then $r_m+r_n=r_m+\rho(x_n,x_{n-1})-\max\{\rho(x_{n-1},x_k)+r_k\setsep 1\le k<n\}\le r_m+\rho(x_n,x_{n-1})-\rho(x_{n-1},x_m)-r_m\le\rho(x_m,x_n)$ for $n>m$ and
\[
\frac{r_{2n}+r_{2n+1}}{\rho(x_{2n},x_{2n+1})}=\frac{r_{2n}+\rho(x_{2n+1},x_{2n})-\max\{\rho(x_{2n},x_k)+r_k\setsep 1\le k\le2n\}}{\rho(x_{2n},x_{2n+1})}>1-\frac1{2n}\to1.
\]
An application of Lemma~\ref{l:seq-isometry} finishes the proof.
\end{proof}

Theorem~\ref{t:linfty} has the following corollary.

\begin{corollary}\label{c:fpp}
Let $M$ be an infinite metric space.
Then $\mc F(M)$ does not have the fixed point property.
\end{corollary}
\begin{proof}
Theorem~\ref{t:linfty} implies that $\mc F(M)^*$ has a subspace isometric to $\ell_\infty$.
It follows that $\mc F(M)$ has a subspace asymptotically isometric to $\ell_1$, see \cite{DGH} or Section~\ref{sec:l1}.
Consequently, $\mc F(M)$ does not have the fixed point property by~\cite[Theorem~2.3]{DLT}.
\end{proof}

\section{\texorpdfstring{Isometric embedding of $\ell_1$ into $\mc F(M)$}{Isometric embedding of l_1 into F(M)}}

In order to prove Theorem~\ref{t:main}, thanks to Theorem~\ref{t:linfty} it remains to consider the ultrametric case (Proposition~\ref{p:ultrametric}).
We embed $\ell_1$ using Lemma~\ref{l:seq-isometry} again.
Further, we show that Lipschitz-free space over a metric space of cardinality at least $3$ is never rotund (Proposition~\ref{p:notRotund}).
This follows from Lemma~\ref{l:twoPoints}, where we compute $\norm{a\de_x+b\de_y}$ for every $x,y\in M$ and $a,b\in\R$ in a general Lipschitz-free space.

First we turn our attention to the result concerning the ultrametric spaces.
In the case that $M$ is unbounded we will use the following lemma, whose idea is the same as in the case of convergent sequence in the proof of Theorem~\ref{t:linfty}.
\begin{lemma}\label{l:unboundedDelta}
Let $(M,\rho)$ be a metric space such that there is a sequence $\{x_n\}\subset M$ with
\[
\rho(x_{2n},x_1)+\rho(x_{2n+1},x_1)-\rho(x_{2n},x_{2n+1})\to\infty.
\]
Then $\mc F(M)$ contains a $1$-complemented subspace isometric to~$\ell_1$.
\end{lemma}
\begin{proof}
Put $\de_n=\frac12\bigl(\rho(x_{2n},x_1)+\rho(x_{2n+1},x_1)-\rho(x_{2n},x_{2n+1})\bigr)$ for $n\in\N$.
By passing to a subsequence we may assume that $\de_m\ge2\rho(x_n,x_1)$ for all $m,n\in\N$, $n<2m$.
We set $r_1=0$, $r_{2n}=\rho(x_{2n},x_1)-\de_n$, and $r_{2n+1}=\rho(x_{2n+1},x_1)-\de_n$.
Now it is easy to see that the assumptions of Lemma~\ref{l:seq-isometry} are satisfied.
Indeed,
\[
r_{2m+i}+r_n\le\rho(x_{2m+i},x_1)-\de_m+\rho(x_n,x_1)\le\rho(x_{2m+i},x_1)-2\rho(x_n,x_1)+\rho(x_n,x_1)\le\rho(x_{2m+i},x_n)
\]
for $2m>n$ and $i\in\{0,1\}$.
\end{proof}

\begin{lemma}\label{l:inc-dec}
Let $(M,\rho)$ be a metric space such that there is a bounded sequence $\{x_n\}\subset M$ of distinct points with the following properties:
$\{\rho(x_k,x_n)\}_{n=k+1}^\infty$ is non-decreasing for each $k\in\N$, $\lim_{n\to\infty}\rho(x_k,x_n)=d_k$, and $\{d_k\}$ is non-increasing.
Then $\mc F(M)$ contains a $1$-complemented subspace isometric to~$\ell_1$.
\end{lemma}
\begin{proof}
Let us denote $d=\lim_{k\to\infty}d_k$.
By passing to a subsequence we may assume that one of the following three cases holds:
a) $\{d_k\}$ is decreasing, or b) $\{d_k\}$ is constant and $\{\rho(x_k,x_n)\}_{n=k+1}^\infty$ are increasing for every $k\in\N$, or c) $\rho(x_k,x_n)=d$ for every $k,n\in\N$, $k\neq n$.

In the case c) we set $r_n=\frac d2$ and apply Lemma~\ref{l:seq-isometry}.

Now consider the case b).
By passing to a further subsequence we may assume that $\rho(x_1,x_2)\ge\frac12d$ and $\rho(x_k,x_{n+1})\ge\frac12d+\frac12\rho(x_{n-1},x_n)$ for $k<n+1$.
We then set $r_1=0$ and $r_{2n}=r_{2n+1}=\frac12\rho(x_{2n},x_{2n+1})$.
Since
\[
r_{2m+i}+r_{2n+j}=\frac12\rho(x_{2m},x_{2m+1})+\frac12\rho(x_{2n},x_{2n+1})\le\frac12d+\frac12\rho(x_{2n},x_{2n+1})\le\rho(x_{2n+j},x_{2n+2})\le\rho(x_{2n+j},x_{2m+i})
\]
for all $m>n$ and $i,j\in\{0,1\}$, it is easy to see that the assumptions of Lemma~\ref{l:seq-isometry} are satisfied.

Finally, we consider the case a).
By passing to a subsequence we may assume that for each $n>k$ the following inequalities hold:
\begin{equation}\label{e:dk-rho-ineq}
d\le d_{k+1}\le\frac34d+\frac14d_k\le\frac12d+\frac12d_k\le\frac14d+\frac34d_k\le\rho(x_k,x_n)\le d_k.
\end{equation}
We set $r_1=0$, $r_{2n}=\rho(x_{2n},x_{2n+1})-\frac12d_{2n+1}$, and $r_{2n+1}=\frac12d_{2n+1}$.
Using~\eqref{e:dk-rho-ineq} we obtain $r_{2m+1}\le r_{2m}\le d_{2m}-\frac12d_{2m+1}\le d_{2m}-\frac12d\le d_{2n+2}-\frac12d$ for all $m>n$.
Hence, using~\eqref{e:dk-rho-ineq} again,
\[
r_{2m+i}+r_{2n}\le d_{2n+2}-\frac12d+\rho(x_{2n},x_{2n+1})-\frac12d_{2n+1}\le\rho(x_{2n},x_{2n+1})\le\rho(x_{2n},x_{2m+i})
\]
and
\[
r_{2m+i}+r_{2n+1}\le d_{2n+2}-\frac12d+\frac12d_{2n+1}\le\frac34d+\frac14d_{2n+1}-\frac12d+\frac12d_{2n+1}=\frac14d+\frac34d_{2n+1}\le\rho(x_{2n+1},x_{2m+i})
\]
for all $m>n$ and $i\in\{0,1\}$.
Also,
\[
r_1+r_{2m+1}\le r_1+r_{2m}\le d_{2m}\le d_2\le\rho(x_1,x_{2m})\le\rho(x_1,x_{2m+1}).
\]
for $m\in\N$.
Thus the assumptions of Lemma~\ref{l:seq-isometry} are satisfied.
\end{proof}

We will use the following property of ultrametric spaces, which is easy to see:
If $x,y,z\in M$ and $\rho(x,y)\neq\rho(y,z)$, then $\rho(x,z)=\max\{\rho(x,y),\rho(y,z)\}$.

\begin{proposition}\label{p:ultrametric}
Suppose that a metric space $(M,\rho)$ contains an infinite ultrametric subspace.
Then the space $\mc F(M)$ contains a $1$\nobreakdash-complemented subspace isometric to~$\ell_1$.
\end{proposition}
\begin{proof}
Let $N\subset M$ be an infinite ultrametric space.
If $N$ is unbounded, then there is a sequence $\{x_n\}\subset N$ such that $\{\rho(x_1,x_n)\}$ is increasing and $\rho(x_1,x_n)\to\infty$.
Since $N$ is ultrametric,  $\rho(x_{2n},x_{2n+1})=\rho(x_1,x_{2n+1})$ for every $n\in\N$ and so we can apply Lemma~\ref{l:unboundedDelta}.

If $N$ is bounded, then it contains a bounded sequence $\{x_n\}$ of distinct points.
By passing to a subsequence we may assume that either $\{\rho(x_1,x_n)\}$ is increasing, or $\{\rho(x_1,x_n)\}$ is decreasing, or $\{\rho(x_k,x_n)\}_{n=k+1}^\infty$ is constant for each $k\in\N$.
We show that in each of these cases we can use Lemma~\ref{l:inc-dec}.

If $\{\rho(x_1,x_n)\}$ is increasing, then $\rho(x_k,x_n)=\max\{\rho(x_1,x_k),\rho(x_1,x_n)\}=\rho(x_1,x_n)$ for every $n>k\ge1$.
Consequently, $\lim_{n\to\infty}\rho(x_k,x_n)=\lim_{n\to\infty}\rho(x_1,x_n)$ and we can apply Lemma~\ref{l:inc-dec}.
If $\{\rho(x_1,x_n)\}$ is decreasing, then $\rho(x_k,x_n)=\max\{\rho(x_1,x_k),\rho(x_1,x_n)\}=\rho(x_1,x_k)$ for $n>k\ge1$.
Consequently, $\lim_{n\to\infty}\rho(x_k,x_n)=\rho(x_1,x_k)$ and, after discarding~$x_1$, we can apply Lemma~\ref{l:inc-dec} again.
Finally, if $\{\rho(x_k,x_n)\}_{n=k+1}^\infty$ is constant for each $k\in\N$, we denote $d_k=\lim_{n\to\infty}\rho(x_k,x_n)=\rho(x_k,x_{k+1})=\rho(x_k,x_{k+2})$.
Then $d_{k+1}=\rho(x_{k+1},x_{k+2})\le\max\{\rho(x_k,x_{k+1}),\rho(x_k,x_{k+2})\}=d_k$.
Hence $\{d_k\}$ is non-increasing and an application of Lemma~\ref{l:inc-dec} finishes the proof.
\end{proof}

\begin{remark}\label{r:unknown_ex}
We do not know whether $\ell_1$ embeds isometrically into the Lipschitz-free space over a general infinite metric space.
Below we show examples of metric spaces where it is impossible to find $r_n$s needed in Lemma~\ref{l:seq-isometry}.
Compare these examples with Lemma~\ref{l:unboundedDelta} and Lemma~\ref{l:inc-dec}.
In all of the examples $M=\{x_n\}$.
The metrics on $M$ are defined for $n>k$ by
\begin{itemize}\setlength\itemsep{0.3em}
\item $\rho(x_k,x_n)=k+n-\frac1k$
\item $\rho(x_k,x_n)=2-\frac1k$
\item $\rho(x_k,x_n)=2-\frac1k+\frac1n$
\item $\rho(x_k,x_n)=2-\frac1k-\frac1{2n}$
\item $\rho(x_k,x_n)=1+\frac1n$
\item $\rho(x_k,x_n)=1+\frac1{2k}+\frac1n$
\end{itemize}
\end{remark}

Now we turn our attention to the proof of Proposition~\ref{p:notRotund}.

\begin{lemma}\label{l:twoPoints}
Let $(M,\rho, 0)$ be a pointed metric space and let $\norm\cdot$ be the canonical norm on $\mc F(M)$.
Then for every $x,y\in M$ and $a,b\in\R$
\[
\norm{a\de_x+b\de_y}=\begin{cases}
\rho(x,0)\abs a+\rho(y,0)\abs b & \text{if $ab\ge0$,}\\
\rho(x,0)\abs a+\bigl(\rho(x,y)-\rho(x,0)\bigr)\abs b & \text{if $ab\le0$ and $\abs b\le\abs a$,}\\
\bigl(\rho(x,y)-\rho(y,0)\bigr)\abs a+\rho(y,0)\abs b & \text{if $ab\le0$ and $\abs b\ge\abs a$.}
\end{cases}
\]
\end{lemma}
See Fig.~\ref{fig:norm}.
\begin{figure}[!ht]
\includegraphics[scale=.8]{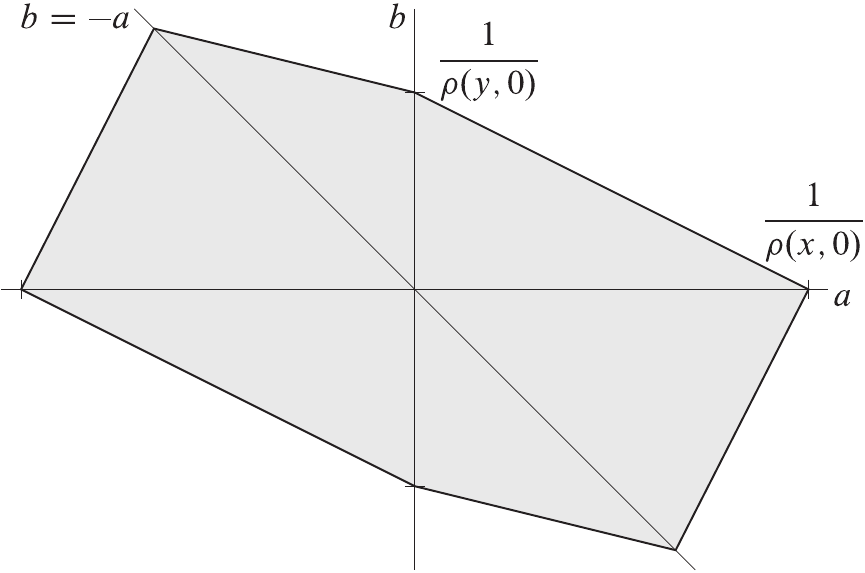}
\caption{Unit ball of $\mc F(M)$ on $\spn\{\de_x,\de_y\}$.}\label{fig:norm}
\end{figure}
This immediately implies Proposition~\ref{p:notRotund}.
Indeed, it suffices to pick any $x,y\in M$, $x\neq y$, $x\neq0$, $y\neq0$, take $u=\frac{\de_x}{\rho(x,0)}$ and $v=\frac{\de_y}{\rho(y,0)}$,
and note that $u\neq v$ since $\de_x$ and $\de_y$ are linearly independent.

\begin{proof}
Let $f\in\Lip_0(M)$ be a function with $\norm f\le 1$.
Then $\abs{f(x)}\le\rho(x,0)$, $\abs{f(y)}\le\rho(y,0)$, and $\abs{f(x)-f(y)}\le\rho(x,y)$.
On the other hand, given $u,v\in\R$ with $\abs u\le\rho(x,0)$, $\abs v\le\rho(y,0)$, and $\abs{u-v}\le\rho(x,y)$
there is $f\in\Lip_0(M)$ with $\norm f\le1$, $f(x)=u$, and $f(y)=v$.
Hence, if we define $g(u,v)=\abs{au+bv}$ for $u,v\in\R$, then
\[
\norm{a\de_x+b\de_y}=\sup_{\substack{f\in\Lip_0(M)\\\norm f\le1}}\abs{\ev f{a\de_x+b\de_y}}=\sup_{\substack{f\in\Lip_0(M)\\\norm f\le1}}\abs{af(x)+bf(y)}=\sup_A g(u,v),
\]
where $A=\bigl\{[u,v]\in\R^2\setsep\abs u\le\rho(x,0),\abs v\le\rho(y,0),\abs{u-v}\le\rho(x,y)\bigr\}$.
Since $-\rho(x,0)\le\rho(y,0)-\rho(x,y)\le\rho(x,0)$ and $-\rho(y,0)\le\rho(x,y)-\rho(x,0)\le\rho(y,0)$, the set $A$ looks like in Fig.~\ref{fig:setA}.
\begin{figure}[!ht]
\includegraphics[scale=.8]{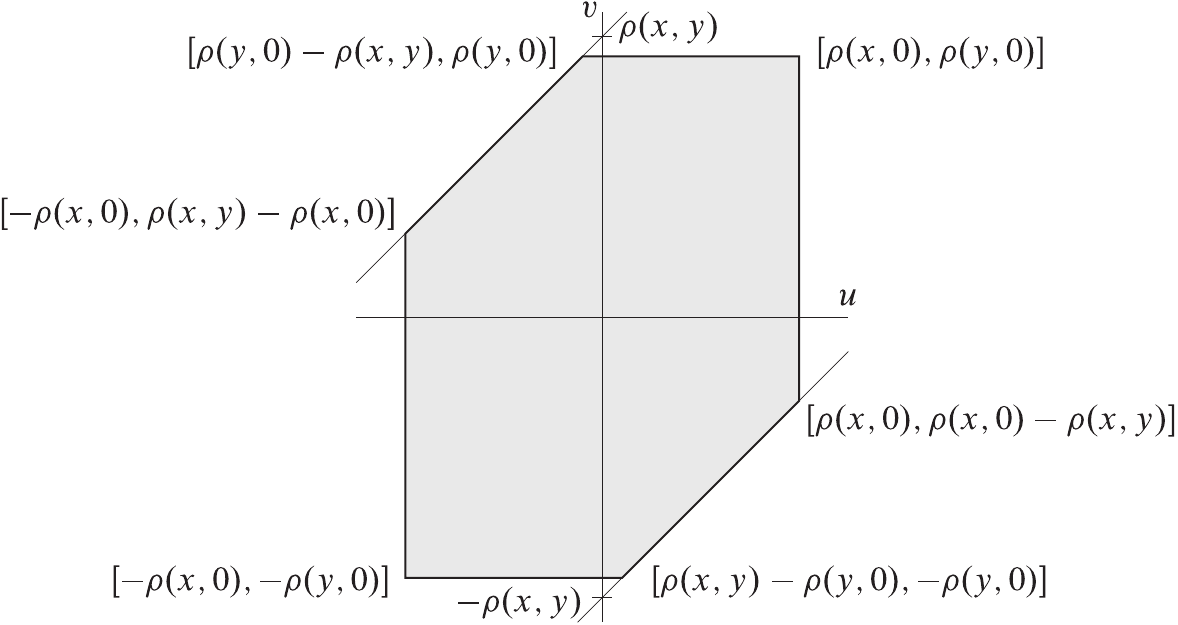}
\caption{The set $A$.}\label{fig:setA}
\end{figure}
The function $g$ is continuous and convex and the set $A$ is convex and compact, and so by the Bauer maximum principle $g$ attains its maximum on $A$ at some extreme point of $A$.
Therefore
\[
\norm{a\de_x+b\de_y}=\max\Bigl\{\abs{a\rho(x,0)+b\rho(y,0)}, \absa{a\rho(x,0)+b\bigl(\rho(x,0)-\rho(x,y)\bigr)}, \absa{a\bigl(\rho(y,0)-\rho(x,y)\bigr)+b\rho(y,0)}\Bigr\}.
\]

The first case is $ab\ge0$.
By the symmetry we can assume that $a\ge0$ and $b\ge0$.
Since $\rho(x,0)-\rho(x,y)\le\rho(y,0)$, we have $a\rho(x,0)+b\bigl(\rho(x,0)-\rho(x,y)\bigr)\le a\rho(x,0)+b\rho(y,0)$.
On the other hand, $2a\rho(x,0)+b\bigl(\rho(y,0)+\rho(x,0)-\rho(x,y)\bigr)\ge0$ and so $-a\rho(x,0)-b\bigl(\rho(x,0)-\rho(x,y)\bigr)\le a\rho(x,0)+b\rho(y,0)$.
Consequently, $\absa{a\rho(x,0)+b\bigl(\rho(x,0)-\rho(x,y)\bigr)}\le a\rho(x,0)+b\rho(y,0)$.
Since this estimate holds for any $x,y\in M$ and $a,b\ge0$, by interchanging $x$ with $y$ and $a$ with $b$ we obtain $\absa{b\rho(y,0)+a\bigl(\rho(y,0)-\rho(x,y)\bigr)}\le a\rho(x,0)+b\rho(y,0)$.
Therefore in this case $\norm{a\de_x+b\de_y}=a\rho(x,0)+b\rho(y,0)$.

Now assume that $ab\le0$.
By the symmetry we can assume that $a\ge0$ and $b\le0$.
The second case is then $-b\le a$.
We have $a\rho(x,0)-b\bigl(\rho(x,y)-\rho(x,0)\bigr)\ge-b\bigl(\rho(x,0)+\rho(x,y)-\rho(x,0)\bigr)\ge0$ and so
$\absa{a\rho(x,0)+b\bigl(\rho(x,0)-\rho(x,y)\bigr)}=a\rho(x,0)-b\bigl(\rho(x,y)-\rho(x,0)\bigr)$.
Further, the inequality $-b\bigl(\rho(x,y)-\rho(x,0)+\rho(y,0)\bigr)\ge0$ implies $a\rho(x,0)+b\rho(y,0)\le a\rho(x,0)-b\bigl(\rho(x,y)-\rho(x,0)\bigr)$
and the inequality $2a\rho(x,0)-b\bigl(\rho(x,y)-\rho(x,0)-\rho(y,0)\bigr)\ge-2b\rho(x,0)-b\bigl(\rho(x,y)-\rho(x,0)-\rho(y,0)\bigr)=-b\bigl(\rho(x,y)+\rho(x,0)-\rho(y,0)\bigr)\ge0$
implies $-a\rho(x,0)-b\rho(y,0)\le a\rho(x,0)-b\bigl(\rho(x,y)-\rho(x,0)\bigr)$.
Consequently, $\abs{a\rho(x,0)+b\rho(y,0)}\le a\rho(x,0)-b\bigl(\rho(x,y)-\rho(x,0)\bigr)$.

Similarly, since $a+b\ge0$, the inequality $(a+b)\rho(x,0)+(a+b)\rho(y,0)\ge(a+b)\rho(x,y)$ implies $-a\bigl(\rho(y,0)-\rho(x,y)\bigr)-b\rho(y,0)\le a\rho(x,0)-b\bigl(\rho(x,y)-\rho(x,0)\bigr)$
and the inequality $(a+b)\rho(y,0)\le(a+b)\rho(x,0)+(a+b)\rho(x,y)\le(a+b)\rho(x,0)+(a-b)\rho(x,y)$ implies
$a\bigl(\rho(y,0)-\rho(x,y)\bigr)+b\rho(y,0)\le a\rho(x,0)-b\bigl(\rho(x,y)-\rho(x,0)\bigr)$.
Consequently, $\absa{a\bigl(\rho(y,0)-\rho(x,y)\bigr)+b\rho(y,0)}\le a\rho(x,0)-b\bigl(\rho(x,y)-\rho(x,0)\bigr)$.
Therefore in this case $\norm{a\de_x+b\de_y}=a\rho(x,0)-b\bigl(\rho(x,y)-\rho(x,0)\bigr)$.

The lase case follows by interchanging $x$ with $y$ and $a$ with $b$.
\end{proof}

\section{\texorpdfstring{Embedding of $\ell_1$ into Banach spaces}{Embedding of l_1 into Banach spaces}}\label{sec:l1}

Here we gather some relations between various types of embedding of $\ell_1$ into a general Banach space.

\begin{definition}\label{d:asymp_l1}
We say that a Banach space $X$ is asymptotically isometric to $\ell_1$ if there are a Schauder basis $\{x_n\}$ of $X$ and a sequence $\{\ve_n\}\subset(0,1)$, $\ve_n\to0$ such that
\[
\sum_{n=1}^\infty(1-\ve_n)\abs{a_n}\le\norma{\sum_{n=1}^\infty a_nx_n}\le\sum_{n=1}^\infty\abs{a_n}
\]
for any $(a_n)\in\ell_1$.
\end{definition}

\begin{definition}
We say that a subspace $Y$ of a Banach space $X$ is asymptotically $1$-complemented in $X$ if there are a Schauder basis $\{x_n\}$ of $Y$ and a projection $P$ of $X$ onto $Y$ such that $\norm{(\Id_Y-P_n)\comp P}\to 1$,
where $P_n$ are the projections associated with $\{x_n\}$.
\end{definition}

The following picture shows relations between various types of embedding.
The black arrows denote implications that hold, the red arrows denote implications that do not hold, the green arrows denote implications that are unknown to us.
\begin{figure}[!ht]
\includegraphics[scale=1]{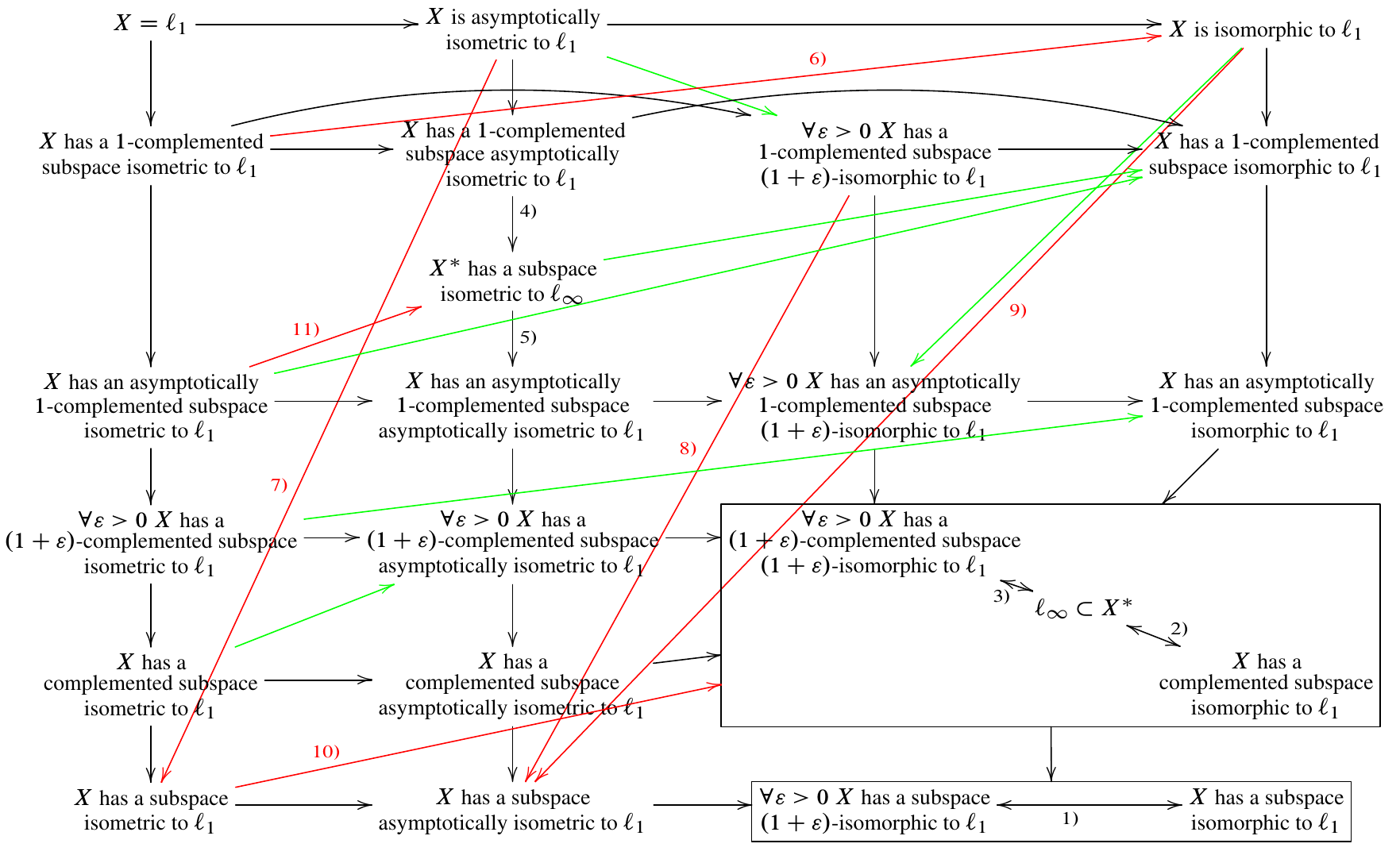}
\end{figure}
We remark that all the counterexamples for the red arrows are separable.

The implications without numbers should be trivial.
For the numbered implications, the arguments follow:

1) James's distortion theorem, \cite[Proposition~2.e.3]{LT}

2) Bessaga, Pełczyński, \cite[Theorem~4.44]{FHHMZ}

3) \cite[Theorem~5]{DRT}

4)
Let $P$ be the projection of $X$ onto $Y\subset X$ of norm $1$, and let $\{x_n\}$ be the basis of $Y$ from Definition~\ref{d:asymp_l1}.
Define $T\colon Y\to\ell_1$ by $T\bigl(\sum_{n=1}^\infty a_nx_n\bigr)=\sum_{n=1}^\infty(1-\ve_n)a_ne_n$ and put $S=T\comp P$.
Then $\norm T\le1$ and hence also $\norm S\le1$.
Further, $\norm{S(x_n)-e_n}=\norm{(1-\ve_n)e_n-e_n)}=\ve_n\to0$.
An application of \cite[Theorem~1]{D} finishes the proof.

5)
By \cite[Theorem~1]{D}, $X$ has a quotient $X/Z$ isometric to $\ell_1$.
Denote by $q\colon X\to X/Z$ the canonical quotient mapping.
Let $\{e_n\}$ be the canonical basis of $X/Z$.
Let $\{\ve_n\}\subset(0,1)$ be a decreasing sequence satisfying $\ve_n\to0$.
For each $n\in\N$ we find $x_n\in X$ such that $q(x_n)=e_n$ and $\norm{x_n}<1+\ve_n$.
Then
\[
\sum_{n=1}^N\abs{a_n}=\norma{\sum_{n=1}^N a_ne_n}=\norma{\sum_{n=1}^N a_nq(x_n)}=\norma{q\left(\sum_{n=1}^N a_nx_n\right)}\le\norma{\sum_{n=1}^N a_nx_n}\le\sum_{n=1}^N\abs{a_n}\norm{x_n}\le\sum_{n=1}^N(1+\ve_n)\abs{a_n}.
\]
Hence $Y=\cspan\{x_n\}$ is asymptotically isometric to $\ell_1$.

Further, define $T\colon X/Z\to Y$ by $T\bigl(\sum_{n=1}^\infty a_ne_n\bigr)=\sum_{n=1}^\infty a_nx_n$ and put $P=T\comp q$.
Then $P$ is clearly a projection from $X$ onto $Y$.
Moreover, denoting by $\{f_n\}$ the functionals biorthogonal to $\{e_n\}$,
\[\begin{split}
\norm{(\Id_Y-P_n)\comp P(x)}&=\norma{(\Id_Y-P_n)\comp T\left(\sum_{i=1}^\infty f_i(q(x))e_i\right)}=\norma{\sum_{i=n+1}^\infty f_i(q(x))x_i}\le\sum_{i=n+1}^\infty(1+\ve_i)\abs{f_i(q(x))}\\
&\le(1+\ve_n)\sum_{i=n+1}^\infty\abs{f_i(q(x))}\le(1+\ve_n)\sum_{i=1}^\infty\abs{f_i(q(x))}=(1+\ve_n)\norm{q(x)}\le(1+\ve_n)\norm x.
\end{split}\]

6) $X=\ell_1\oplus\ell_2$

7) Let $X=(\ell_1,\tnorm\cdot)$, where $\tnorm x^2=\norm x_1^2+\sum_{n=1}^\infty\frac1{2^n}x_n^2$.
Then $\tnorm\cdot$ is rotund, so $X$ does not contain a subspace isometric to $\ell_1$.

8) and 9) \cite[Example~2.8]{DLT}

10) $X=C([0,1])$.
By Pełczyński's theorem \cite[Theorem~7.6]{HMVZ} every non-reflexive infinite-dimensional complemented subspace of $C([0,1])$ contains $c_0$.

11) This follows from the proof of \cite[Proposition~4]{JR} together with \cite[Theorem~1]{D}.

\bigskip

We suspect that none of the green implications hold.

\end{document}